\documentclass[11pt,reqno,twoside]{article}
\newcommand{\rad}{\operatorname{rad}}

\usepackage[hang]{footmisc}
\usepackage{lipsum}

\usepackage{cmap} 

\usepackage[T1]{fontenc}
\usepackage[utf8]{inputenc}
\usepackage{graphicx}
\usepackage{array}
\usepackage{placeins}
\usepackage{enumerate}
\usepackage{booktabs}
\usepackage{algcompatible}
\usepackage[most]{tcolorbox}
\usepackage{mdframed}

\usepackage{verbatim}
\newcommand{\comments}[1]{}

\usepackage{soul}

\usepackage{setspace}

\let\counterwithin\relax  
\usepackage{lmodern} 

\usepackage{comment}

\usepackage{bm} 

\usepackage{bbold}

\usepackage{amsmath,amsbsy,amsgen,amscd,amsthm,amsfonts,amssymb}

\usepackage{float}       
\usepackage{algorithm}
\usepackage{algpseudocode}

\usepackage[centering,top=1.1in,bottom=1.4in,left=1in,right=1in]{geometry}

\usepackage[sf,bf,compact]{titlesec}
\usepackage[mathscr]{euscript}

\usepackage{enumitem}
\usepackage[dvipsnames]{xcolor}

\definecolor{dark-gray}{gray}{0.3}
\definecolor{dkgray}{rgb}{.4,.4,.4}
\definecolor{dkblue}{rgb}{0,0,.5}
\definecolor{medblue}{rgb}{0,0,.75}
\definecolor{rust}{rgb}{0.5,0.1,0.1}

\usepackage{url}
\usepackage[colorlinks=true]{hyperref}
\hypersetup{linkcolor=dkblue}    
\hypersetup{citecolor=rust}      
\hypersetup{urlcolor=rust}     

\usepackage[final]{microtype} 

\newtheoremstyle{myThm} 
    {\topsep}                    
    {\topsep}                    
    {\itshape}                   
    {}                           
    {\sffamily\bfseries}                   
    {.}                          
    {.5em}                       
    {}  

\newtheoremstyle{myRem} 
    {\topsep}                    
    {\topsep}                    
    {}                   
    {}                           
    {\sffamily}                   
    {.}                          
    {.5em}                       
    {}  

\newtheoremstyle{myDef} 
    {\topsep}                    
    {\topsep}                    
    {}                   
    {}                           
    {\sffamily\bfseries}                   
    {.}                          
    {.5em}                       
    {}  

\theoremstyle{myThm}
\newtheorem{theorem}{Theorem}[section]
\newtheorem{lemma}[theorem]{Lemma}
\newtheorem{proposition}[theorem]{Proposition}

\theoremstyle{myRem}
\newtheorem{remarkx}[theorem]{Remark}
 \newenvironment{remark}
  {\pushQED{\qed}\remarkx}
  {\popQED\endremarkx}
\usepackage{fancyhdr}
\let\originalleft\left
\let\originalright\right
\renewcommand{\left}{\mathopen{}\mathclose\bgroup\originalleft}
\renewcommand{\right}{\aftergroup\egroup\originalright}

\usepackage{mathtools}
\mathtoolsset{centercolon}  

\providecommand{\mathbbm}{\mathbb} 

\newcommand{\R}{\mathbbm{R}}
\newcommand{\E}{\mathbbm{E}}

\newcommand{\mcN}{\mathcal{N}}

\newcommand{\tr}{\mathrm{Tr}}   

\newcommand{\Var}{\mathsf{Var}}

\definecolor{mygreen}{rgb}{0.13,0.55,0.13}

\newcommand{\iid}{\stackrel{\text{i.i.d.}}{\sim}}

\usepackage[font = small, margin=30pt]{caption}

\usepackage[]{algorithm}
\usepackage{enumerate}

\usepackage{graphicx}

\usepackage{authblk}
\usepackage{chngcntr}
\usepackage{mathrsfs} 
\counterwithin{table}{section}
\counterwithin{algorithm}{section}

\usepackage[normalem]{ulem}

\usepackage{tikz}
\usetikzlibrary{shapes,arrows,arrows.meta, decorations.pathmorphing,backgrounds,positioning,fit,petri}
\usetikzlibrary{calc}
\usetikzlibrary{matrix}
\usetikzlibrary{backgrounds}
\usetikzlibrary{shapes.geometric}
\usetikzlibrary{patterns}
\usepackage{pgfplots}
\pgfplotsset{compat=newest}
\usepgfplotslibrary{groupplots}
\pgfdeclarelayer{background}
\pgfdeclarelayer{bBack}
\pgfsetlayers{bBack,background,main}
\usetikzlibrary{fit}

\usepackage{subcaption}

\title{Concentration Inequalities for Sample Cross-Covariances}

\author{Jiaheng Chen and Daniel Sanz-Alonso}

\date{University of Chicago}

\vspace{.5in}

\makeatletter\@addtoreset{section}{part}\makeatother%
\numberwithin{equation}{section}

\newcommand{\upperRomannumeral}[1]{\uppercase\expandafter{\romannumeral#1}}

\begin{document}
\maketitle 


\renewcommand{\thefootnote}{\fnsymbol{footnote}}


\vspace{-2em}
\abstract{This paper establishes sharp dimension-free concentration and expectation
bounds for the deviation of a sample cross-covariance matrix from its mean.
For sub-Gaussian random vectors, we prove a high-probability operator-norm
bound governed by the effective ranks of the two marginal covariance matrices.
In the Gaussian case, we prove a matching expectation lower bound,
allowing arbitrary correlation between the two random vectors.}

\bigskip

\section{Introduction}

Sample covariance matrices are among the most fundamental objects in probability and statistics. Concentration inequalities for sample
covariances play a central role in covariance estimation, principal component
analysis, regression, empirical process theory, and many related problems. 
A now-standard message is that the relevant dimension parameter is often not
the ambient dimension, but rather the effective rank of the covariance matrix $\Sigma,$ defined by
\[
r(\Sigma)=\frac{\tr(\Sigma)}{\|\Sigma\|},
\]
where \(\tr(\Sigma)\) denotes the trace and \(\|\Sigma\|\) denotes the operator norm. For instance, \cite{koltchinskii2017concentration} proved that, if \(X_1,\ldots,X_N\) are independent copies of a centered
sub-Gaussian random vector \(X\) with covariance matrix \(\Sigma\), then the
sample covariance satisfies, for every $u\ge 1$,  with probability at least $1-\exp(-u)$,
\[
\bigg\|
\frac1N\sum_{i=1}^N X_iX_i^\top-\Sigma
\bigg\|
\lesssim
\|\Sigma\|
\bigg(
\sqrt{\frac{r(\Sigma)+u}{N}}
+
\frac{r(\Sigma)+u}{N}
\bigg).
\]

This paper studies the corresponding problem for sample cross-covariances.  Sample cross-covariance matrices arise just as naturally in multivariate statistics and
data analysis. They are central in multivariate regression, canonical correlation analysis, instrumental variables, and
operator learning problems where one observes paired input-output data. 

Let \((X_i,Y_i)_{i=1}^N\) be independent copies of a centered random vector
\((X,Y)\in\mathbb R^{d_X+d_Y}\), with marginal covariance
matrices \(\Sigma_X\) and \(\Sigma_Y\). We are interested in sharp
high-probability bounds for
\[
\bigg\| \frac{1}{N}\sum_{i=1}^N X_iY_i^\top-\E XY^\top \bigg\|.
\] 
Although this deviation is formally similar
to the corresponding deviation for the
sample covariance, its sharp tail behavior is not yet fully understood.
Indeed, 
the off-diagonal structure and  possible dependence between  \(X\) and \(Y\) make the analysis of sample cross-covariances more delicate. Our goal is to put sample cross-covariances on the same theoretical footing as sample covariances.

\section{Main Result}
The Orlicz \(\psi_2\) norm and the
\(L_p\) norm of a real-valued random variable \(Z\) are defined by
\[
\|Z\|_{\psi_2}
:=
\inf\left\{
t>0:\mathbb{E}\exp(Z^2/t^2)\le 2
\right\},
\qquad
\|Z\|_{L_p} := \big(\mathbb{E}|Z|^p\big)^{1/p}.
\]
We say that a random vector \(X \in \R^d\)  is sub-Gaussian if there exists \(K>0\) such
that
\[
\|\langle X,v\rangle\|_{\psi_2}
\le K \|\langle X,v\rangle\|_{L_2},
\qquad
\text{for any } v\in \mathbb{S}^{d-1},
\]
where $\mathbb{S}^{d-1}$ denotes the unit sphere in $\R^d.$ 
The smallest such constant $K>0$ is called the sub-Gaussian constant of $X$. The sub-Gaussian constant of a centered Gaussian random variable is $K= \sqrt{8/3}$. For positive sequences $\{a_n\}, \{b_n\}$, we write $a_n \lesssim b_n$ to denote that, for some constant $c >0$, $a_n \le c b_n.$ If both $a_n \lesssim b_n$ and $b_n \lesssim a_n$ hold, we write $a_n \asymp b_n.$ Throughout the paper, all covariance matrices are allowed to be singular, but are assumed to be nonzero.

Our main result establishes operator-norm concentration bounds for sample cross-covariances.

\begin{theorem}\label{thm:main1}
Let \((X_i,Y_i)_{i=1}^N\) be i.i.d. copies of a centered random vector
\((X,Y)\in\mathbb R^{d_X+d_Y}\). Assume that \(X\) and \(Y\) are
sub-Gaussian random vectors with covariance matrices \(\Sigma_X\) and
\(\Sigma_Y\), and sub-Gaussian constants \(K_X,K_Y\). Set
\[
r_{X}:=\frac{\tr(\Sigma_X)}{\|\Sigma_X\|},
\qquad
r_{Y}:=\frac{\tr(\Sigma_Y)}{\|\Sigma_Y\|}.
\]
For every \(u\ge 1\), it holds with probability at least \(1-\exp(-u)\) that
\[
\bigg\|\frac{1}{N}\sum_{i=1}^{N} X_i Y_i^{\top}-\E X Y^{\top}\bigg\| \lesssim K_{X}K_{Y}(\|\Sigma_{X}\|\|\Sigma_Y\|)^{1/2} \bigg(\sqrt{\frac{r_{X}+r_{Y}+u}{N}}+ \frac{\sqrt{(r_{X}+u)(r_{Y}+u)}}{N}\bigg).
\]
As a corollary,
 \begin{align*}
\E \bigg\|\frac{1}{N}\sum_{i=1}^{N} X_i Y^{\top}_i-\E X Y^{\top}\bigg\| \lesssim K_{X}K_{Y} (\|\Sigma_{X}\|\|\Sigma_Y\|)^{1/2} \bigg(\sqrt{\frac{r_{X}+r_{Y}}{N}}+ \frac{\sqrt{r_{X}r_{Y}}}{N}\bigg).
 \end{align*}
Moreover, if $(X,Y)$ is jointly Gaussian, then
\begin{align*}
\E \bigg\|\frac{1}{N}\sum_{i=1}^{N} X_i Y^{\top}_i-\E XY^{\top}\bigg\| \asymp (\|\Sigma_{X}\|\|\Sigma_Y\|)^{1/2} \bigg(\sqrt{\frac{r_{X}+r_{Y}}{N}}+ \frac{\sqrt{r_{X}r_{Y}}}{N}\bigg).
\end{align*} 
\end{theorem}

We make a few remarks on Theorem~\ref{thm:main1}. 

\begin{remark}\leavevmode
\begin{enumerate}
    \item The upper bound in Theorem~\ref{thm:main1} only requires \(X\) and \(Y\) to be
individually sub-Gaussian, while allowing arbitrary dependence between them.
The matching lower bound is proved in the jointly Gaussian case, again allowing
arbitrary dependence between \(X\) and \(Y\). In particular, when \(X=Y\), Theorem \ref{thm:main1} recovers the dimension-free concentration inequalities for sample
covariances established in
\cite[Theorems 4 and 9]{koltchinskii2017concentration}.

    \item In the case where \(X\) and \(Y\) are independent,
    \cite[Lemma A.3]{ghattas2022non} obtained a high-probability upper bound by
    applying the covariance bound of \cite{koltchinskii2017concentration} to
    the concatenated vector \((X,Y)\). This argument, however, does not capture
    the product structure of the cross-covariance fluctuation, such as the
    factors
    \(\sqrt{\|\Sigma_X\|\|\Sigma_Y\|}\) and \(\sqrt{r_X r_Y}\). In the more general setting where \(X\) and \(Y\) may be arbitrarily
dependent, \cite[Theorem 2.1]{chen2026sharp2} obtained an in-expectation
upper bound for sub-Gaussian random vectors as a corollary of more general
concentration results for asymmetric tensors, without capturing the sharp tail behavior.    The new contributions of
Theorem~\ref{thm:main1} are twofold. First, it gives a sharp high-probability
upper bound for sample cross-covariances; the proof is short and transparent, combining a standard quadratic
empirical process inequality with a scaled polarization argument and an
optimization over the scaling parameter. Second, in the jointly Gaussian
case, it establishes a matching lower bound allowing arbitrary correlation
between \(X\) and \(Y\).

    \item The sharp upper and lower bounds in Theorem~\ref{thm:main1} show
    that, in the Gaussian case, the size of the fluctuation is essentially
    insensitive to the strength of the dependence between \(X\) and \(Y\).
    To see the intuition, write
    \[
    \bigg\|\frac{1}{N}\sum_{i=1}^{N} X_i Y_i^{\top}-\E X Y^{\top}\bigg\|
    =
    \sup_{\|v\|_2 \le 1,\, \|h\|_2\le 1}
    \left|
    \frac{1}{N}\sum_{i=1}^N
    \langle X_i,v\rangle \langle Y_i,h\rangle
    -
    \E \langle X,v\rangle \langle Y,h\rangle
    \right|.
    \]
   For fixed directions \(v\) and \(h\),
    \((\langle X,v\rangle,\langle Y,h\rangle)\) is jointly Gaussian, and
    Isserlis' formula \cite{isserlis1918formula} gives
    \[
    \Var\bigl(\langle X,v\rangle\langle Y,h\rangle\bigr)
    =
    \E\langle X,v\rangle^2\,\E\langle Y,h\rangle^2
    +
    \bigl(\E\langle X,v\rangle\langle Y,h\rangle\bigr)^2 .
    \]
    Therefore, by Cauchy--Schwarz,
    \[
    \E\langle X,v\rangle^2\,\E\langle Y,h\rangle^2
    \le
    \Var\bigl(\langle X,v\rangle\langle Y,h\rangle\bigr)
    \le
    2\E\langle X,v\rangle^2\,\E\langle Y,h\rangle^2 .
    \]
    Thus the variance of the product \(\langle X,v\rangle\langle Y,h\rangle\)
    is comparable to
    \(\E\langle X,v\rangle^2\,\E\langle Y,h\rangle^2\), with universal
    constants, independently of the covariance
    \(\E\langle X,v\rangle\langle Y,h\rangle\). This explains why, in the
    Gaussian case, the fluctuation scale is governed by the marginal
    covariances of \(X\) and \(Y\), rather than by the size of the
    cross-covariance. Theorem~\ref{thm:main1} makes this intuition rigorous in
    operator norm, where one takes the supremum over all unit vectors \(v\)
    and \(h\).
\end{enumerate}
\end{remark}

\begin{proof}[Proof of Theorem \ref{thm:main1}]
    The high-probability upper bound in the sub-Gaussian case follows from Theorem \ref{thm:main2}.   Write
\[
X=\Sigma_X^{1/2}Z,\qquad Y=\Sigma_Y^{1/2}W,
\]
where \(Z\) and \(W\) are isotropic sub-Gaussian random vectors. We apply Theorem \ref{thm:main2} with \[
T=\Sigma_{X}^{1/2}\mathbb{S}^{d_{X}-1}, \qquad S=\Sigma_{Y}^{1/2}\mathbb{S}^{d_{Y}-1}.
\]
Using that
    \[
    \rad(T)=\|\Sigma_X\|^{1/2}, \qquad \gamma(T)\le \tr(\Sigma_X)^{1/2}=\|\Sigma_X\|^{1/2}r_X^{1/2},
    \]
    and
      \[
    \rad(S)=\|\Sigma_Y\|^{1/2}, \qquad \gamma(S)\le \tr(\Sigma_Y)^{1/2}=\|\Sigma_Y\|^{1/2}r_Y^{1/2},
    \]
    gives the desired bound.
    The in-expectation upper bound follows by integrating the tail bound. The matching lower bound in the jointly Gaussian case follows from Proposition \ref{prop:gaussian_lower_bound}.
\end{proof}

\section{Upper Bound}

We say that a random vector in \(\mathbb{R}^d\) is
isotropic if its covariance matrix is the identity matrix \(I_d\). Next, let
\(g \sim \mcN(0,I_d)\) be the standard multivariate Gaussian random vector. For a set \(T\subset \R^d\), we define its radius, Gaussian complexity,
and stable dimension by
\begin{align}\label{eq:geometric_quantity}
\rad(T):=\sup_{v\in T}\|v\|_2,\qquad
\gamma(T):=\E \sup_{v\in T} |\langle g,v \rangle|,\qquad
d(T):=\left(\frac{\gamma(T)}{\rad(T)}\right)^2.
\end{align}
We refer to
\cite[Chapter 7]{vershynin2018high} for background on these geometric
quantities.

\begin{theorem}\label{thm:main2}
Let \((Z_i,W_i)_{i=1}^N\) be i.i.d. copies of a centered random
vector \((Z,W)\in \mathbb R^{d_Z+d_W}\). Assume that \(Z\) and \(W\) are isotropic
sub-Gaussian random vectors with sub-Gaussian constants \(K_Z\) and \(K_W\).  Let
\(T\subset \mathbb R^{d_Z}\) and \(S\subset \mathbb R^{d_W}\) be bounded sets. For every \(u\ge 1\), it holds with probability at least  \(1-\exp(-u)\)  that
\begin{align}
&\sup_{v\in T,\, h\in S}
\left|
\frac{1}{N}\sum_{i=1}^N
\langle Z_i,v\rangle \langle W_i,h\rangle
-
\E \langle Z,v\rangle \langle W,h\rangle
\right| \nonumber\\
&\lesssim K_Z K_W \rad(T)\rad(S)\left(
\sqrt{\frac{d(T)+d(S)+u}{N}}+\frac{\sqrt{(d(T)+u)(d(S)+u)}}{N}\right).
\end{align}
\end{theorem}


Our proof of Theorem~\ref{thm:main2} is based on the following standard quadratic
empirical process bound (see e.g., \cite[Corollary 5.7]{dirksen2015tail}, \cite[Theorem 1]{bednorz2014}, \cite[Theorem 8]{koltchinskii2017concentration}), combined with a scaled polarization argument and an
optimization over the scaling parameter. The scaling parameter in the polarization identity is used to preserve the
separate complexities of the two index sets \(T\) and \(S\), as well as the
dependence on the confidence level \(u\). This is essential for obtaining the
product term in the final bound.

Recall that, for a metric space \((T,d)\), Talagrand's \(\gamma_2\)-functional \cite{talagrand2022upper} is defined by
\[
\gamma_2(T,d)
:=
\inf_{\{T_n\}}
\sup_{t\in T}
\sum_{n\ge 0} 2^{n/2} d(t,T_n),
\]
where \(d(t,T_n):=\inf_{s\in T_n}d(t,s)\), and the infimum is taken over all
admissible sequences \(\{T_n\}_{n\ge 0}\) of subsets of \(T\), that is, sequences
satisfying $|T_0|=1$ and $|T_n|\le 2^{2^n}$ for $n\ge 1$.

\begin{lemma}\label{lem:quad_emp}
Let $\xi, \xi_1, \ldots, \xi_N \iid \mu$ be a sequence of random variables on a probability space $(\Omega, \mu)$, and let \(\mathcal F\) be a class of functions on $(\Omega, \mu)$. Set
\[
\Delta(\mathcal F):=\sup_{f\in\mathcal F}\|f\|_{\psi_2},
\qquad
\Gamma(\mathcal F):=\gamma_2(\mathcal F,\|\cdot\|_{\psi_2}).
\]
  Then, for every
\(u\ge1\), with probability at least \(1-\exp(-u)\),
\[
\sup_{f\in\mathcal F}
\left|
\frac1N\sum_{i=1}^N f^2(\xi_i)-\E f^2(\xi)
\right|
\lesssim
\frac{\Gamma(\mathcal F)\Delta(\mathcal F)}{\sqrt N}
+
\frac{\Gamma^2(\mathcal F)}{N}
+
\Delta^2(\mathcal F)
\left(\sqrt{\frac uN}+\frac uN\right).
\]
\end{lemma}

Now we are ready to prove Theorem \ref{thm:main2}.

\begin{proof}[Proof of Theorem~\ref{thm:main2}]
For \(\lambda>0\), define the two function classes
\[
\mathcal F_\lambda^+
:=
\left\{
(z,w)\mapsto \langle z,v\rangle+\lambda\langle w,h\rangle:
v\in T,\ h\in S
\right\},
\]
and
\[
\mathcal F_\lambda^-
:=
\left\{
(z,w)\mapsto \langle z,v\rangle-\lambda\langle w,h\rangle:
v\in T,\ h\in S
\right\}.
\]
Since \(Z\) and \(W\) are centered, all functions in
\(\mathcal F_\lambda^\pm\) are centered. Moreover, by the triangle inequality
for the \(\psi_2\)-norm,
\[
\Delta(\mathcal F_\lambda^\pm)
\le
K_Z\rad(T)+\lambda K_W\rad(S).
\]
Indeed, for \(v\in T\) and \(h\in S\),
\[
\|\langle Z,v\rangle\pm \lambda\langle W,h\rangle\|_{\psi_2}
\le
\|\langle Z,v\rangle\|_{\psi_2}
+
\lambda\|\langle W,h\rangle\|_{\psi_2}
\le
K_Z\|v\|_2+\lambda K_W\|h\|_2.
\]

We next bound the \(\gamma_2\)-functional of \(\mathcal F_\lambda^\pm\).
For
\(f_{v,h}^{\pm},f_{v',h'}^{\pm}\in\mathcal F_\lambda^\pm\),
\[
\|f_{v,h}^{\pm}-f_{v',h'}^{\pm}\|_{\psi_2}
\le
K_Z\|v-v'\|_2+\lambda K_W\|h-h'\|_2.
\]
Hence admissible chains for \(T\) and \(S\) can be combined,  with a one-level
shift to preserve admissibility,  to form an
admissible chain for \(\mathcal F_\lambda^\pm\), yielding
\[
\Gamma(\mathcal F_\lambda^\pm)
\lesssim
K_Z \gamma_2(T,\|\cdot\|_2)+\lambda K_W\gamma_2(S,\|\cdot\|_2).
\]
By the majorizing measure theorem \cite[Theorem 2.10.1]{talagrand2022upper},
\[
\gamma_2(T,\|\cdot\|_2)\lesssim\gamma(T),
\qquad
\gamma_2(S,\|\cdot\|_2)\lesssim\gamma(S),
\]
where \(\gamma(T)\) and \(\gamma(S)\) are the Gaussian complexities of \(T\) and
\(S\), as defined in \eqref{eq:geometric_quantity}.
Consequently,
\[
\Gamma(\mathcal F_\lambda^\pm)
\lesssim
K_Z\gamma(T)+\lambda K_W\gamma(S).
\]
Set
\[
A:=K_Z\gamma(T),\qquad B:=K_W\gamma(S),
\qquad
C:=K_Z\rad(T),\qquad D:=K_W\rad(S).
\]
We may assume that \(C>0\) and \(D>0\), as the case \(\rad(T)=0\) or \(\rad(S)=0\) is
trivial.
Then
\[
\Gamma(\mathcal F_\lambda^\pm)\lesssim A+\lambda B,
\qquad
\Delta(\mathcal F_\lambda^\pm)\le C+\lambda D.
\]

For every \(v\in T\) and \(h\in S\), the scaled polarization identity gives
\[
\langle z,v\rangle\langle w,h\rangle
=
\frac{
(\langle z,v\rangle+\lambda\langle w,h\rangle)^2
-
(\langle z,v\rangle-\lambda\langle w,h\rangle)^2
}{4\lambda},\qquad \lambda>0.
\]
Therefore,
\[
\begin{aligned}
&\sup_{v\in T,h\in S}
\left|
\frac1N\sum_{i=1}^N
\langle Z_i,v\rangle\langle W_i,h\rangle
-
\E\langle Z,v\rangle\langle W,h\rangle
\right| \\
&\le
\frac{1}{4\lambda}
\sup_{f\in\mathcal F_\lambda^+}
\left|
\frac1N\sum_{i=1}^N f^2(Z_i,W_i)-\E f^2(Z,W)
\right| +
\frac{1}{4\lambda}
\sup_{f\in\mathcal F_\lambda^-}
\left|
\frac1N\sum_{i=1}^N f^2(Z_i,W_i)-\E f^2(Z,W)
\right|.
\end{aligned}
\]
Applying Lemma~\ref{lem:quad_emp} to \(\mathcal F_\lambda^+\) and
\(\mathcal F_\lambda^-\), and absorbing the union bound into the constants,
we obtain with probability at least \(1-\exp(-u)\),
\[
\begin{aligned}
&\sup_{v\in T,h\in S}
\left|
\frac1N\sum_{i=1}^N
\langle Z_i,v\rangle\langle W_i,h\rangle
-
\E\langle Z,v\rangle\langle W,h\rangle
\right| \\
&\lesssim
\frac1\lambda
\left[
\frac{(A+\lambda B)(C+\lambda D)}{\sqrt N}
+
\frac{(A+\lambda B)^2}{N}
+
(C+\lambda D)^2
\left(\sqrt{\frac uN}+\frac uN\right)
\right].
\end{aligned}
\]

For every fixed $u\ge 1$, we optimize over \(\lambda\). Write the right-hand side as
\[
\frac{P}{\lambda}+Q\lambda+R,
\]
where
\[
P:=
\frac{AC}{\sqrt N}
+
\frac{A^2}{N}
+
C^2\left(\sqrt{\frac uN}+\frac uN\right),
\]
\[
Q:=
\frac{BD}{\sqrt N}
+
\frac{B^2}{N}
+
D^2\left(\sqrt{\frac uN}+\frac uN\right),
\]
and
\[
R:=
\frac{AD+BC}{\sqrt N}
+
\frac{2AB}{N}
+
2CD\left(\sqrt{\frac uN}+\frac uN\right).
\]
Choosing
\[
\lambda=\sqrt{\frac{P}{Q}},
\]
we get
\[
\frac{P}{\lambda}+Q\lambda=2\sqrt{PQ}.
\]

It remains to bound \(\sqrt{PQ}\). Let
\[
x:=\frac{A}{C\sqrt N},
\qquad
y:=\frac{B}{D\sqrt N},
\qquad
r:=\sqrt{\frac uN}.
\]
Then
\[
P=C^2(x+x^2+r+r^2),
\qquad
Q=D^2(y+y^2+r+r^2).
\]
Since
\[
x+x^2+r+r^2\le (x+r)+(x+r)^2,
\]
and similarly for \(y\), the elementary inequality
\[
\sqrt{(a+a^2)(b+b^2)}
\lesssim a+b+ab,\qquad a,b\ge0,
\]
with \(a=x+r\) and \(b=y+r\), gives
\[
\sqrt{PQ}
\lesssim
CD\left(
x+y+r+xy+rx+ry+r^2
\right).
\]
Returning to \(A,B,C,D\), this yields
\[
\sqrt{PQ}
\lesssim
\frac{AD+BC+\sqrt u\,CD}{\sqrt N}
+
\frac{(A+\sqrt u\,C)(B+\sqrt u\,D)}{N}.
\]
The same expression also dominates \(R\), up to universal constants. Hence
\[
\begin{aligned}
\sup_{v\in T,h\in S}
\left|
\frac1N\sum_{i=1}^N
\langle Z_i,v\rangle\langle W_i,h\rangle
-
\E\langle Z,v\rangle\langle W,h\rangle
\right| \lesssim
\frac{AD+BC+\sqrt u\,CD}{\sqrt N}
+
\frac{(A+\sqrt u\,C)(B+\sqrt u\,D)}{N}.
\end{aligned}
\]
Substituting the definitions of \(A,B,C,D\), we obtain
\[
\begin{aligned}
&\sup_{v\in T,h\in S}
\left|
\frac1N\sum_{i=1}^N
\langle Z_i,v\rangle\langle W_i,h\rangle
-
\E\langle Z,v\rangle\langle W,h\rangle
\right|\\
&\lesssim
K_ZK_W
\bigg(
\frac{
\gamma(T)\rad(S)+\gamma(S)\rad(T)+\sqrt u\,\rad(T)\rad(S)
}{\sqrt N}
\\
&\hspace{3.8cm}
+
\frac{
(\gamma(T)+\sqrt u\,\rad(T))
(\gamma(S)+\sqrt u\,\rad(S))
}{N}
\bigg).
\end{aligned}
\]
Finally, using
\[
\gamma(T)=\rad(T)\sqrt{d(T)},
\qquad
\gamma(S)=\rad(S)\sqrt{d(S)},
\]
we get
\[
\begin{aligned}
&\sup_{v\in T,h\in S}
\left|
\frac1N\sum_{i=1}^N
\langle Z_i,v\rangle\langle W_i,h\rangle
-
\E\langle Z,v\rangle\langle W,h\rangle
\right|\\
&\lesssim
K_ZK_W\rad(T)\rad(S)
\left(
\sqrt{\frac{d(T)+d(S)+u}{N}}
+
\frac{\sqrt{(d(T)+u)(d(S)+u)}}{N}
\right).
\end{aligned}
\]
This proves the theorem.
\end{proof}

\section{Lower Bound}
We now prove the matching lower bound for Gaussian data in
Theorem~\ref{thm:main1}. Importantly, the result does not require \(X\) and
\(Y\) to be independent; arbitrary correlation between them is allowed.

\begin{proposition}\label{prop:gaussian_lower_bound}
Let \((X_i,Y_i)_{i=1}^N\) be i.i.d. copies of a centered Gaussian random
vector $(X,Y) \in \R^{d_X+d_Y}$.
Then
\begin{align*}
\E \bigg\|\frac{1}{N}\sum_{i=1}^{N} X_i Y^{\top}_i-\E XY^{\top}\bigg\| \gtrsim (\|\Sigma_{X}\|\|\Sigma_Y\|)^{1/2} \bigg(\sqrt{\frac{r_{X}+r_{Y}}{N}}+ \frac{\sqrt{r_{X}r_{Y}}}{N}\bigg).
\end{align*} 
\end{proposition}

\begin{proof}[Proof of Proposition \ref{prop:gaussian_lower_bound}]
Write the covariance matrix of the jointly Gaussian vector \((X,Y)\) as
\[
\begin{pmatrix}
\Sigma_X & \Sigma_{XY}\\
\Sigma_{YX} & \Sigma_Y
\end{pmatrix}.
\]

By the Gaussian regression formula \cite[Section 2.1.7]{rue2005gaussian},  there exist a matrix
\(L=\Sigma_{YX}\Sigma_X^\dagger\) and a centered Gaussian vector \(Z\),
independent of \(X\), such that
\[
Y=LX+Z.
\]
Here \(\Sigma_X^\dagger\) denotes the Moore--Penrose pseudoinverse of \(\Sigma_X\). Moreover,
\[
\Sigma_Y=L\Sigma_XL^\top+\Sigma_Z,
\]
where \(\Sigma_Z=\mathbb E ZZ^\top\). Hence, substituting \(Y_i=LX_i+Z_i\)
for each \(i\), we obtain
\[
\frac1N\sum_{i=1}^N X_iY_i^\top-\mathbb E XY^\top
=
\underbrace{\bigg(\frac{1}{N}\sum_{i=1}^N
X_iX_i^\top -\mathbb E XX^\top 
\bigg)L^\top}_{:=A_N}+\underbrace{\frac{1}{N}\sum_{i=1}^N X_iZ_i^\top}_{:=B_N}.
\]
Here  \((X_i,Z_i)_{i=1}^N\) are i.i.d. copies of \((X,Z)\) with \(X_i\) independent
of \(Z_i\).

We first note that the two terms $A_N$ and $B_N$ cannot cancel in expectation. Indeed,
conditioning on \(X_1,\ldots,X_N\), the matrix \(A_N\) is deterministic and
\(B_N\) is a centered Gaussian random matrix. Thus, by Jensen's inequality,
\[
\mathbb E_Z\|A_N+B_N\|
\ge
\left\|A_N+\mathbb E_Z B_N\right\|
=
\|A_N\|.
\]
On the other hand, conditionally on \(X_1,\ldots,X_N\), the random matrices
\(B_N\) and \(-B_N\) have the same distribution. Hence, by the triangle inequality,
\begin{align*}
\E_Z \|A_N+B_N\|
&=
\frac{1}{2}\E_Z\|A_N+B_N\|
+
\frac{1}{2}\E_Z\|A_N-B_N\| \\
&\ge
\frac{1}{2}\E_Z
\big\|(A_N+B_N)-(A_N-B_N)\big\| \\
&=
\E_Z\|B_N\|.
\end{align*}
Taking expectation over \(X_1,\ldots,X_N\), we obtain
\[
\E\|A_N+B_N\|
\ge
\max\{\E\|A_N\|,\E\|B_N\|\}
\ge
\frac12\left(\E\|A_N\|+\E\|B_N\|\right).
\]
Therefore it remains to lower bound the two pieces separately. 
  
Set
\[
P:=L\Sigma_XL^\top,
\qquad
Q:=\Sigma_Z.
\]
Then \(P,Q\succeq0\) and \(P+Q=\Sigma_Y\).

For the correlated term $A_N$, Lemma \ref{lemma:cov_times_L_lower} yields, if $P \neq 0,$
\[
\mathbb E\|A_N\|
\gtrsim
(\|\Sigma_X\|\|P\|)^{1/2}
\bigg(
\sqrt{\frac{r_X+r_P}{N}}
+
\frac{\sqrt{r_Xr_P}}{N}
\bigg)=\|\Sigma_X\|^{1/2}
\bigg(
\sqrt{\frac{\|P\|r_X+\tr(P)}{N}}
+
\frac{\sqrt{r_X\tr(P)}}{N}
\bigg),
\]
where $ r_P:=\tr(P)/\|P\|.$ If $P=0,$ we simply bound $\mathbb E\|A_N\| \ge 0.$ 

For the independent residual term
\(B_N\), the independent Gaussian cross-covariance bound \cite[Theorem 2.1]{chen2026sharp2} gives
\[
\mathbb E\|B_N\|
\asymp
(\|\Sigma_X\|\|Q\|)^{1/2}
\bigg(
\sqrt{\frac{r_X+r_Q}{N}}
+
\frac{\sqrt{r_Xr_Q}}{N}
\bigg)=\|\Sigma_X\|^{1/2}
\bigg(
\sqrt{\frac{\|Q\|r_X+\tr(Q)}{N}}
+
\frac{\sqrt{r_X\tr(Q)}}{N}
\bigg),
\]
where $r_Q:=\tr(Q)/\|Q\|.$


It remains to combine the two scales. 
Since \(P,Q\succeq0\) and \(P+Q=\Sigma_Y\), we have 
\[
\|P\|+\|Q\|\ge \|\Sigma_Y\|,\qquad \tr(P)+\tr(Q)=\tr(\Sigma_Y)
.
\]
Consequently,
\[
\sqrt{\|P\|}+\sqrt{\|Q\|}
\ge
\sqrt{\|\Sigma_Y\|}, \qquad
\sqrt{\tr(P)}+\sqrt{\tr(Q)}
\ge
\sqrt{\tr(\Sigma_Y)}.
\]
Therefore,
\[
\E\|A_N+B_N\|\ge \frac{1}{2}(\E \|A_N\|+\E \|B_N\|)
\gtrsim
\|\Sigma_X\|^{1/2}
\bigg(
\sqrt{\frac{\|\Sigma_Y\| r_X+\tr(\Sigma_Y)}{N}}
+
\frac{\sqrt{r_X\tr(\Sigma_Y)}}{N}
\bigg).
\]
Using $r_Y=\tr(\Sigma_Y)/\|\Sigma_Y\|$,
the right-hand side is equivalent, up to universal constants, to
\[
(\|\Sigma_X\|\|\Sigma_Y\|)^{1/2}
\bigg(
\sqrt{\frac{r_X+r_Y}{N}}
+
\frac{\sqrt{r_Xr_Y}}{N}
\bigg).
\]
This proves the desired lower bound.
\end{proof}

\begin{lemma}\label{lemma:cov_times_L_lower}
Let \((X_i)_{i=1}^N\) be i.i.d. copies of a centered Gaussian random
vector \(X\in\mathbb R^{d_X}\) with covariance matrix \(\Sigma_X\), and let
\(L\) be a fixed matrix such that $        P:=L\Sigma_XL^\top $ is nonzero. Then
\[
\E \left\|
\bigg( \frac1N\sum_{i=1}^N
X_iX_i^\top -\E XX^\top\bigg) L^\top
\right\|
\gtrsim
(\|\Sigma_X\|\|P\|)^{1/2}
\bigg(
\sqrt{\frac{r_X+r_P}{N}}
+
\frac{\sqrt{r_Xr_P}}{N}
\bigg),
\]
where
\[
r_X:=\frac{\tr(\Sigma_X)}{\|\Sigma_X\|},
\qquad
r_P:=\frac{\tr(P)}{\|P\|}.
\]
\end{lemma}

\begin{proof}[Proof of Lemma \ref{lemma:cov_times_L_lower}]

We first prove the contribution of order \(N^{-1}\). By conditioning on
\(X_1\) and averaging over \(X_2,\ldots,X_N\), Jensen's inequality gives
\[
\E \left\|
\bigg( \frac1N\sum_{i=1}^N X_iX_i^\top-\Sigma_X\bigg)L^\top
\right\|
\ge
\frac{1}{N}
\E\|(XX^\top-\Sigma_X)L^\top\|.
\]
By the triangle inequality,
\[
\E\|(XX^\top-\Sigma_X)L^\top\|
\ge
\E\|XX^\top L^\top\|-\|\Sigma_XL^\top\|.
\]
Since
\[
        \|XX^\top L^\top\|=\|X\|\,\|LX\|,
\]
we lower bound \(\E\|X\|\,\|LX\|\) using the Gaussian correlation
inequality \cite{royen2014simple,latala2017royen}. Indeed, for \(s,t\ge0\), the sets
\[
        \{x:\|x\|\le s\},
        \qquad
        \{x:\|Lx\|\le t\}
\]
are symmetric and convex. Hence the Gaussian correlation inequality implies
\[
\mathbb P(\|X\|\le s,\|LX\|\le t)
\ge
\mathbb P(\|X\|\le s)\mathbb P(\|LX\|\le t).
\]
Equivalently,
\[
\mathbb P(\|X\|>s,\|LX\|>t)
\ge
\mathbb P(\|X\|>s)\mathbb P(\|LX\|>t).
\]
Using the layer-cake representation,
\[
\E\|X\|\,\|LX\|
=
\int_0^\infty\int_0^\infty
\mathbb P(\|X\|>s,\|LX\|>t)\,ds\,dt,
\]
we obtain
\[
        \E\|X\|\,\|LX\|
        \ge
        \E\|X\|\,\E\|LX\|.
\]
Moreover, for centered Gaussian vectors,
\[
        \E\|X\|\asymp \sqrt{\tr(\Sigma_X)},
        \qquad
        \E\|LX\|\asymp \sqrt{\tr(L\Sigma_XL^\top)}
        =
        \sqrt{\tr(P)} .
\]
Therefore
\[
        \E\|X\|\,\|LX\|
        \ge
        c_0\sqrt{\tr(\Sigma_X)\tr(P)} .
\]
On the other hand,
\[
        \|\Sigma_XL^\top\|
        =
        \|\Sigma_X^{1/2}\Sigma_X^{1/2}L^\top\|
        \le
        \|\Sigma_X\|^{1/2}\|\Sigma_X^{1/2}L^\top\|
        =
        \|\Sigma_X\|^{1/2}\|L\Sigma_XL^\top\|^{1/2}.
\]
Thus, writing
\begin{equation}\label{def:AandR}
          A:=(\|\Sigma_X\|\|P\|)^{1/2},
        \qquad
        R:=\sqrt{\tr(\Sigma_X)\tr(P)}
        =
        A\sqrt{r_Xr_P},  
\end{equation}
we have shown
\begin{align}\label{eq:lower_bound_aux1}
\E \left\|
\bigg( \frac1N\sum_{i=1}^N X_iX_i^\top-\Sigma_X\bigg)L^\top
\right\|
\ge
\frac{c_0 R-A}{N}.
\end{align}

We next prove the contribution of order \(N^{-1/2}\). Choose
\(w\in\mathbb{S}^{d_Y-1}\) such that
\[
        w^\top P w=\|P\|.
\]
Set
\[
        \ell:=L^\top w,
        \qquad
        \eta:=\langle X,\ell\rangle,
        \qquad
        \sigma^2:=\E\eta^2=\ell^\top\Sigma_X\ell=\|P\|.
\]
Then
\[
\E \left\|
\bigg( \frac1N\sum_{i=1}^N X_iX_i^\top-\Sigma_X\bigg)L^\top
\right\|
\ge
\E\left\|
\bigg( \frac1N\sum_{i=1}^N X_iX_i^\top-\Sigma_X\bigg)\ell
\right\|=
\E\left\|
\frac1N\sum_{i=1}^N
\bigl(X_i\eta_i-\E X\eta\bigr)
\right\|.
\]
We orthogonalize \(X\) with respect to \(\eta\). Write
\[
        X=\frac{\Sigma_X\ell}{\sigma^2}\eta+X^\perp,
\]
where \(X^\perp\) is centered Gaussian and independent of \(\eta\), with covariance
\[
        \Sigma_\perp
        =
        \Sigma_X-\frac{\Sigma_X\ell\ell^\top\Sigma_X}{\sigma^2}.
\]
Indeed,
\[
        \E X^\perp \eta
        =
        \E X\eta-\frac{\Sigma_X\ell}{\sigma^2}\E\eta^2
        =
        \Sigma_X\ell-\Sigma_X\ell=0,
\]
and uncorrelated jointly Gaussian random vectors are independent. 
Thus
\[
        X\eta-\E X\eta
        =
        \frac{\Sigma_X\ell}{\sigma^2}(\eta^2-\sigma^2)
        +
        X^\perp\eta .
\]
Consequently,
\[
\frac1N\sum_{i=1}^N
\bigl(X_i\eta_i-\E X\eta\bigr)
=
U_N+V_N,
\]
where
\[
        U_N:=
        \frac{\Sigma_X\ell}{\sigma^2}
        \frac1N\sum_{i=1}^N(\eta_i^2-\sigma^2),
        \qquad
        V_N:=
        \frac1N\sum_{i=1}^N X_i^\perp\eta_i .
\]
Conditionally on \((\eta_i)_{i=1}^N\), \(U_N\) is deterministic and \(V_N\)
is a centered Gaussian vector. Hence, by Jensen's inequality,
\[
        \E_{X^\perp}\|U_N+V_N\|\ge \|U_N\|.
\]
Also, by symmetry of \(V_N\),
\[
        \E_{X^\perp}\|U_N+V_N\|
        =
        \frac12\E_{X^\perp}\|U_N+V_N\|
        +
        \frac12\E_{X^\perp}\|U_N-V_N\|
        \ge
        \E_{X^\perp}\|V_N\|.
\]
Therefore,
\[
        \E\|U_N+V_N\|
        \ge
        \frac12\E\|U_N\|+\frac12\E\|V_N\|.
\]

We now estimate the two terms. Since \(\eta_i/\sigma\) are i.i.d. standard
Gaussians,
\[
        \E\|U_N\|
        =
        \frac{\|\Sigma_X\ell\|}{\sigma^2}
        \E\bigg|
        \frac1N\sum_{i=1}^N(\eta_i^2-\sigma^2)
        \bigg|
        \asymp
        \frac{\|\Sigma_X\ell\|}{\sqrt N}.
\]
Conditionally on \((\eta_i)_{i=1}^N\),
\[
        V_N
        \sim
        \mcN\bigg(
        0,\,
        \frac1{N^2}\sum_{i=1}^N\eta_i^2\,\Sigma_\perp
        \bigg).
\]
For a centered Gaussian vector \(G\) with covariance matrix \(\Gamma\), we use
\[
        \E\|G\|\asymp \sqrt{\tr(\Gamma)}.
\]
Hence
\[
        \E\|V_N\|
        \asymp
        \frac{1}{N}
        \E\left(\sum_{i=1}^N\eta_i^2\right)^{1/2}
        \sqrt{\tr(\Sigma_\perp)}
        \asymp
        \frac{\sigma}{\sqrt N}\sqrt{\tr(\Sigma_\perp)}.
\]
It follows that
\[
\E\left\|
\frac1N\sum_{i=1}^N
\bigl(X_i\eta_i-\E X\eta\bigr)
\right\|
\gtrsim
\frac1{\sqrt N}
\left(
        \|\Sigma_X\ell\|
        +
        \sigma\sqrt{\tr(\Sigma_\perp)}
\right).
\]
Now
\[
        \tr(\Sigma_\perp)
        =
        \tr(\Sigma_X)
        -
        \frac{\|\Sigma_X\ell\|^2}{\sigma^2}.
\]
Therefore, setting \(a:=\|\Sigma_X\ell\|/\sigma\), we have
\begin{equation}\label{eq:elementary}
        \|\Sigma_X\ell\|
        +
        \sigma\sqrt{\tr(\Sigma_\perp)}
        =
        \sigma
        \left(
        a+
        \sqrt{\tr(\Sigma_X)-a^2}
        \right)
        \ge
        \sigma\sqrt{\tr(\Sigma_X)}.
\end{equation}
Since \(\sigma^2=\|P\|\), we obtain
\begin{align}\label{eq:lower_bound_aux2}
\E \left\|
\bigg( \frac1N\sum_{i=1}^N X_iX_i^\top-\Sigma_X\bigg)L^\top
\right\|
\gtrsim
\sqrt{\frac{\|P\|\tr(\Sigma_X)}{N}}
=
(\|\Sigma_X\|\|P\|)^{1/2}
\sqrt{\frac{r_X}{N}}.
\end{align}

We now prove the other \(N^{-1/2}\) contribution. Choose
\(v\in\mathbb S^{d_X-1}\) such that
\[
        v^\top\Sigma_Xv=\|\Sigma_X\|.
\]
Set
\[
        \xi:=\langle X,v\rangle,
        \qquad
        \tau^2:=\E\xi^2=\|\Sigma_X\|.
\]
Since
\[
        \left\|
        \bigg( \frac1N\sum_{i=1}^N X_iX_i^\top-\Sigma_X\bigg)L^\top
        \right\|
        =
        \left\|
        L\bigg( \frac1N\sum_{i=1}^N X_iX_i^\top-\Sigma_X\bigg)
        \right\|,
\]
we have
\[
\E \left\|
\bigg( \frac1N\sum_{i=1}^N X_iX_i^\top-\Sigma_X\bigg)L^\top
\right\|
\ge
\E\left\|
L\bigg( \frac1N\sum_{i=1}^N X_iX_i^\top-\Sigma_X\bigg)v
\right\|.
\]
The right-hand side equals
\[
\E\left\|
\frac1N\sum_{i=1}^N
\bigl(LX_i\xi_i-\E LX\,\xi\bigr)
\right\|.
\]
We now orthogonalize \(LX\) with respect to \(\xi\). Write
\[
        LX=\frac{L\Sigma_Xv}{\tau^2}\xi+Z,
\]
where \(Z\) is centered Gaussian and independent of \(\xi\), with covariance
\[
        P_\perp
        =
        P-\frac{L\Sigma_Xvv^\top\Sigma_XL^\top}{\tau^2}.
\]
Repeating the preceding argument gives
\[
\E\left\|
\frac1N\sum_{i=1}^N
\bigl(LX_i\xi_i-\E LX\,\xi\bigr)
\right\|
\gtrsim
\frac1{\sqrt N}
\left(
        \|L\Sigma_Xv\|
        +
        \tau\sqrt{\tr(P_\perp)}
\right).
\]
Since
\[
        \tr(P_\perp)
        =
        \tr(P)-\frac{\|L\Sigma_Xv\|^2}{\tau^2},
\]
the same elementary inequality as in \eqref{eq:elementary} gives
\[
        \|L\Sigma_Xv\|
        +
        \tau\sqrt{\tr(P_\perp)}
        \ge
        \tau\sqrt{\tr(P)}.
\]
Thus
\begin{align}\label{eq:lower_bound_aux3}
\E \left\|
\bigg( \frac1N\sum_{i=1}^N X_iX_i^\top-\Sigma_X\bigg)L^\top
\right\|
\gtrsim
\sqrt{\frac{\|\Sigma_X\|\tr(P)}{N}}
=
(\|\Sigma_X\|\|P\|)^{1/2}
\sqrt{\frac{r_P}{N}}.
\end{align}
Combining \eqref{eq:lower_bound_aux2} and \eqref{eq:lower_bound_aux3} yields
\begin{align}\label{eq:lower_bound_aux4}
\E \left\|
\bigg( \frac1N\sum_{i=1}^N X_iX_i^\top-\Sigma_X\bigg)L^\top
\right\|
\gtrsim
(\|\Sigma_X\|\|P\|)^{1/2}
\sqrt{\frac{r_X+r_P}{N}}.
\end{align}

Finally, we combine the \(N^{-1}\) estimate \eqref{eq:lower_bound_aux1} with the
\(N^{-1/2}\) estimate \eqref{eq:lower_bound_aux4} using the convex-combination trick. From
\eqref{eq:lower_bound_aux4} and \(r_X,r_P\ge1\), we have 
\begin{align*}
\E \left\|
\bigg( \frac1N\sum_{i=1}^N X_iX_i^\top-\Sigma_X\bigg)L^\top
\right\|
\ge
c_1 A\sqrt{\frac{r_X+r_P}{N}}
\ge
c_1\frac{A}{N}.
\end{align*}
Together with \eqref{eq:lower_bound_aux1}, for every \(t\in[0,1]\),
\[
\E \left\|
\bigg( \frac1N\sum_{i=1}^N X_iX_i^\top-\Sigma_X\bigg)L^\top
\right\|
\ge
t\frac{c_0R-A}{N}
+
(1-t)c_1\frac{A}{N}.
\]
Choose \(t>0\) sufficiently small so that
\[
        (1-t)c_1-t\ge0.
\]
Then, using \eqref{def:AandR},
\begin{equation}\label{eq:lower_bound_aux5}
    \E \left\|
\bigg( \frac1N\sum_{i=1}^N X_iX_i^\top-\Sigma_X\bigg)L^\top
\right\|
\ge
\frac{tc_0R}{N}
\gtrsim (\|\Sigma_X\|\|P\|)^{1/2}
 \frac{\sqrt{r_Xr_P}}{N}.
\end{equation}

Combining  \eqref{eq:lower_bound_aux4} and \eqref{eq:lower_bound_aux5}, and recalling that
\(P=L\Sigma_XL^\top\), we obtain
\[
\E \left\|
\bigg( \frac1N\sum_{i=1}^N X_iX_i^\top-\Sigma_X\bigg)L^\top
\right\|
\gtrsim
(\|\Sigma_X\|\|L\Sigma_XL^\top\|)^{1/2}
\bigg(
\sqrt{\frac{r_X+r_P}{N}}
+
\frac{\sqrt{r_Xr_P}}{N}
\bigg).
\]
This proves the lemma. 
\end{proof}

\section*{Acknowledgments}
The authors were partly funded by the NSF CAREER award DMS-2237628.

\bibliographystyle{siam} 
\bibliography{references}

\end{document}